\documentclass[12pt,reqno]{amsart}

\usepackage{amssymb}
\usepackage{amsthm}

\usepackage{amsmath,amssymb,amsfonts,amsopn,color}
\usepackage{pdfsync}
\usepackage{amssymb}
\usepackage{enumerate}
\usepackage{tikz}
\usepackage{pgfplots}
\usepackage{mathrsfs}

\usepackage[normalem]{ulem}

\usepackage{anysize}
\marginsize{2.5cm}{2.5cm}{2cm}{3cm}

\newtheorem{theorem}{Theorem}
\newtheorem{lemma}[theorem]{Lemma} 
\newtheorem{proposition}[theorem]{Proposition}
\newtheorem{corollary}[theorem]{Corollary} 

\theoremstyle{definition}  
\newtheorem{defn}[theorem] {Definition}
\newtheorem*{example}{Example}
\newtheorem*{remark}{Remark} 
\theoremstyle{remark}

\newcommand{\R}{\mathbb R}
\newcommand{\C}{\mathbb C}

\providecommand{\scal}[2]{\langle#1,#2\rangle}

\newcommand{\cA}{{\mathcal A}}

\newcommand{\cD}{{\mathcal D}}

\newcommand{\cF}{{\mathcal F}}

\newcommand{\cI}{{\mathcal I}}

\newcommand{\cQ}{{\mathcal Q}}

\newcommand{\cS}{{\mathcal S}}

\newcommand{\cW}{{\mathcal W}}

\newcommand{\D}{{\rm d}}
\newcommand{\E}{{\rm e}}

\begin{document}

\title{Radon transform intertwines shearlets and wavelets}

\author{F.~Bartolucci}
\address{F.~Bartolucci, Dipartimento di Matematica, Universit\`a di Genova,  Via
  Dodecaneso 35, Genova,   Italy  }
\email{bartolucci@dima.unige.it}

\author{F.~De~Mari}
\address{F.~De~Mari, Dipartimento di Matematica, Universit\`a di Genova,  Via
  Dodecaneso 35, Genova,   Italy  }
\email{demari@dima.unige.it}
\author{E.~De Vito}
\address{E.~De Vito ,Dipartimento di Matematica, Universit\`a di Genova,  Via
  Dodecaneso 35, Genova,   Italy  }
\email{devito@dima.unige.it}

\author{F.~Odone}
\address{F.~Odone, Dibris, Universit\`a di Genova,  Via
  Dodecaneso 35, Genova,   Italy  }
\email{odne@unige.it}

\begin{abstract}
We prove that the unitary affine Radon transform intertwines the quasi-regular representation of a
class of semidirect products, built by shearlet dilation groups and translations, and the tensor product 
of a standard wavelet representation with a wavelet-like representation. This yields a formula for 
shearlet coefficients that involves only integral transforms applied to the affine Radon transform of the signal,
thereby opening new perspectives in the inversion of the Radon transform.
\end{abstract}

\keywords{shearlets; wavelets; Radon transform}


\maketitle

\section{Introduction} 

The use of wavelets in signal analysis and computer vision has proved
almost optimal for one-dimensional signals in many ways, and the
mathematics behind classical wavelets has reached a high degree of
elaboration. In higher dimensions, however, the picture is less
  clear and  this partially explains the huge class of
representations that has been introduced over the years to handle high
dimensional problems, such as directional wavelets \cite{anmu96},
ridgelets \cite{cado99}, curvelets \cite{cado04}, wavelets with
composite dilations \cite{gulaLiwewi04}, contourlets \cite{dove05},
shearlets \cite{lalikuwe05}, reproducing groups of the symplectic
group \cite{grsa08}, Gabor ridge functions \cite{grsa08} and mocklets
\cite{dede2013} -- to name a few.

Among them, shearlets stand out because of their ability to efficiently capture
anisotropic features, to provide optimal sparse representations, to
detect singularities and to be stable against noise,
see \cite{kula12} for an overview and a complete list of references.
From the purely mathematical  perspective, their construction is based on 
the well-established theory of  square-integrable representations
\cite{dumo76}, just as wavelets are,
and because of this many powerful mathematical tools are available. As far as
applications are concerned, their effectiveness has been tested primarily  in image 
processing, where many efficient algorithms have been designed using them
(see \cite{kula12,duodde15} and  the website
{\tt http://www.shearlab.org/} for further details and references).

Thus, in some sense, shearlets behave for high-dimensional signals as wavelets do for 1D-signal, 
 and it is therefore natural to try to understand if the many strong connections are
 a consequence of some general mathematical principle. 
 
The purpose of this paper is to address this issue, and to give a partial
answer, showing  that  the link
between the shearlet transform and wavelets is  the unitary Radon transform
in affine coordinates,
because 
it actually intertwines the shearlet representation 
with a tensor product of two wavelet representations. This fact can be exploited
to show that by carefully choosing
the mother shearlet it is possible to obtain the classical shearlet coefficients as 
a sequence of operations performed on the 
Radon transform of the signal, 
namely  a  one-dimensional wavelet transform, with respect to the ``sliding'' coordinate
that parametrizes all the hyperplanes parallel to a given one (for the two dimensional case see Fig.~\ref{fig:radonaffine}),  followed by a convolution 
operator with a scale-dependent filter in the variables of the hyperplane. As the shearlet transform admits an inversion formula, it is in principle possible to invert the Radon transform of a given signal by means of it and
the aforementioned operations.

For two-dimensional signals, our results, which have been announced
in \cite{efff17}, can be described as we now explain. In order to
formulate them precisely, we recall the definition of the three main
ingredients, namely wavelets, shearlets and the Radon
transform.  The wavelet group is $\R\rtimes\R^\times$ with law
$(b,a)(b',a')=(b+ab',aa')$.  The square integrable
wavelet representation $W$ acts on $L^2(\R)$ by
\[
W_{b,a}\psi(x) = |a|^{-1/2} \psi\bigl(\frac{x-b}{a}\bigr)
\]
and the wavelet transform, defined by
$ \mathcal W_\psi f(b,a)=\scal{f}{\psi_{b,a}}$,
is a  multiple of  an isometry  from $L^2(\R)$ to $L^2(\R\rtimes\R^\times)$ provided that $\psi$
satisfies the Calder\'on condition,  see~\eqref{calderon} below.
Next, denote by 
${\mathbb S}$ the (parabolic) shearlet group, namely $\R^2\rtimes(\R\rtimes\R^\times)$ with multiplication
\[ 
(b,s,a)(b',s',a')  = (b+N_sA_ab',s+ |a|^{1/2} s',aa')\]
where
\begin{alignat*}{1}
A_a & = a
    \begin{bmatrix}
      1 & 0 \\
      0 & |a|^{- 1/2} 
    \end{bmatrix},
    \qquad
    \quad N_s=
    \begin{bmatrix}
      1 & -s \\
      0 & 1
    \end{bmatrix}
  \end{alignat*}
and where the vectors are understood as column vectors. 
The group ${\mathbb S}$ acts on $L^2(\R^2)$ via the shearlet representation, namely 
    \[ 
    S_{b,s,a}f\,(x)=
    |a|^{- 3/4}f(A_{a}^{-1} N_{s}^{-1}(x-b)).
    \]
The shearlet transform  is then  
$\cS_\psi f(b,s,a)= \scal{f}{S_{b,s,a}\psi}$, and
is a multiple of an isometry provided that an admissibility
condition  on $\psi$ is satisfied  \cite{dahlke2008,kula09},
  see~\eqref{eqn:admvect} below. 
Finally, the Radon transform  in affine coordinates  of  a signal
$f\in L^1(\R^2)$ is the function  
${\mathcal R}^{\rm aff} f:\R^2\to \C$  defined by
\[
{\mathcal R}^{\rm aff} f (v,t)=  \int\limits_{\R}f(t-v y,y) \,\D y,
\qquad (v,t)\in\R^2.
\] 
An important fact is that it is possible to define a version of ${\mathcal R}^{\rm aff}$  as a unitary map on $L^2(\R^2)$. First, it is necessary to compose it with the Riesz-type operator ${\mathcal I}$ that we now describe.  Its natural domain is  the dense subspace of $L^2(\R^2)$ 
\[
\mathcal D = \Bigl\{ g\in L^2(\R^2):
    \int\limits_{\R^2}
  |\xi_2| \, |\widehat{g}(\xi_1,\xi_2)|^2\D\xi_1\D\xi_2<+\infty   
   \Bigr\}, 
   \]
 where   $\widehat{g}$  denotes the Fourier transform of $g$. The densely
defined, self-adjoint unbounded operator ${\mathcal I}: \mathcal D\to   L^2(\R^2)$ is defined  by
\[
\widehat{(\mathcal I g)} (\xi_1,\xi_2)
= |\xi_2|^{\frac{1}{2}} \,\widehat{g}(\xi_1,\xi_2), \qquad (\xi_1,\xi_2)\in\R^2,
\]
{\em i.e.} a Fourier multiplier in the  second variable. It is not hard to show
that for all $f$ in the dense subspace of $ L^2(\R^2)$ 
\[
\cA= \Bigl\{ f\in L^1(\R^d)\cap L^2(\R^2): 
\int_{\R^2}\frac{|\widehat{f} (\xi_1,\xi_2)|^2}{|\xi_1|}\D\xi_1\D\xi_2<+\infty   \Bigr\}, 
   \]
 the Radon
transform ${\mathcal R}^{\rm aff}  f$ belongs to $\mathcal D$ and that the map
$f \longmapsto \mathcal I {\mathcal R}^{\rm aff}  f $
from $\cA$ to $L^2(\R^2)$ extends to a unitary
map, denoted by $\mathcal Q$, from $L^2(\R^2)$ onto itself. 

In the two dimensional case our main formula reads now

  \begin{equation}
  {\mathcal Q} \, S_{b,s,a} f=
\left( W_{s,|a|^{1/2}}\otimes
      \operatorname{I}\right) W_{(1,\mathbf{v})\cdot b,a}\,  {\mathcal Q}
    f\label{result2D} 
  \end{equation}
where the meaning of the dummy variable $\mathbf{v}$ is
\[
W_{(1,\mathbf{v})\cdot b,a} g(v,t)
= |a|^{-\frac12}  \,g\left(v,\frac{t- (1,v)\cdot b}{a}\right).
\]
Our second most important result is the formula
\begin{equation}
  \label{secondmain2D}
  \mathcal S_\psi f(x,y,s,a) 
 = |a|^{-\frac{3}{4}}  \int\limits_{\R}
\mathcal W_{\chi_1 }
  \big({\mathcal R}^{\rm aff}  f(v,\bullet) \big)(x+v y,a) \overline{ 
  \phi_2 \Big(\frac{v-s}{|a|^{1/2}}\Big)}\,\D v,
\end{equation}
provided that $f\in L^1(\R^d)\cap L^2(\R^d)$ and $\psi$ is of the form
\[
\widehat{\psi}(\xi_1,\xi_2)= \widehat{\psi_1}(\xi_1)\widehat{\psi_2}\Bigl(\frac{\xi_2}{\xi_1}\Bigr).
\]
The $1D$-wavelet  $\chi_1$ and  the $1D$-filter $\phi_2$  are
related to the shearlet admissible vector $\psi$ by the following relations
\begin{alignat}{1}
  \widehat{\chi_1}(\xi_1) & =|\xi_1|\widehat{\psi_1}(\xi_1) \label{eq:1}\\
{\phi_2}(\xi_2/\xi_1) & =\widehat{\psi_2}(\xi_2/\xi_1).\label{eq:2}
\end{alignat}
The first equality shows that $2\pi\chi_1=H\psi_1'$ is the Hilbert transform
$H$ of the weak derivative of $\psi_1$.

Equation~\eqref{secondmain2D} shows that for any signal $f$ in
$L^1(\R^d)\cap L^2(\R^d)$ the shearlet coefficients can be computed by
means of three {\em classical} trasforms. Indeed, in order to obtain 
$ \mathcal S_\psi f(x,y,s,a)$ one can:
\begin{enumerate}[a)]
\item compute the Radon transform ${\mathcal R}^{\rm aff} f(v,t)$ of
  the original signal $f$;
\item apply the  wavelet transform with respect to the variable $t$
  \begin{equation}
    G(v,b,a) = \mathcal W_{\chi_1 } \big({\mathcal R}^{\rm aff} f(v,\cdot)
    \big)(b,a),
    \label{eq:7}
  \end{equation}
where $\chi_1$ is given by~\eqref{eq:1};
\item convolve   the result  with the scale-dependent filter
  \[\Phi_a(v)= \overline{\phi_2
      \left(-\frac{v}{|a|^{1/2}}\right)},\]
where $\phi_2$ is given by~\eqref{eq:2} and the convolution is computed with respect to  the variable $v$,
that is
\[
    \mathcal S_\psi f(x,y,s,a) = \left(G(\bullet ,x+\bullet
      \,y,a)*\Phi_a\right)(s).
\]
\end{enumerate}
Finally, since $S$ is a square-integrable representation, there
is a reconstruction formula, namely 
\begin{equation}
f = \int_{\mathbb S} \mathcal S_\psi f(x,y,s,a)\,\, S_{x,y,s,a}\psi
\,\,\frac{\D x \D y \D s \D a}{|a|^3} ,\label{eq:3}
\end{equation}
where the integral converges in the weak sense. Note that
$\mathcal S_\psi f$ depends on $f$ only through its Radon
transform ${\mathcal R}^{\rm aff} f$, see~\eqref{secondmain2D}. The
above equation allows to 
reconstruct an unknown signal $f$ from its Radon transform by computing
the shearlet coefficients by means of~\eqref{eq:3}.

The fact that the Radon
transform does play a prominent role in this circle of ideas is not new. Indeed,
it is known that ridgelets are constructed via wavelet analysis in the
Radon domain \cite{fast09},  Gabor frames are defined as the
directionally-sensitive Radon transforms \cite{grsa08}, discrete
shearlet frames are used to invert the Radon transform
\cite{coeagula10} and the Radon transform is at the root of the proof
that shearlets are able to detect the wavefront set of a 2D signal
\cite{gr11}. 

Our contribution is to clarify this relation from the point of view of
non-commutative harmonic analysis.  We are actually able to prove a
rather general result, Theorem~\ref{main}, which generalizes
\eqref{result2D} and holds for the class of groups that were
introduced by F\"uhr \cite{fu98,futo16} and that are
known as {\it shearlet dilation groups}.

The paper is organized as follows. In Section~\ref{prelim} we present
in full detail all the various ingredients, namely the groups, the
representations, the Radon transform and the unitary extensions that
need to be defined. In Section~\ref{main} we state and prove the main
results.


\section{Preliminaries} \label{prelim}
\subsection{Notation}
We briefly introduce the notation. We set
$\R^{\times}=\R\setminus\{0\}$. The Euclidean norm of a vector
$v\in\R^d$ is denoted by $|v|$ and its scalar product with $w\in\R^d$ 
by $v\cdot w$. For any $p\in[1,+\infty]$ we denote
by $L^p(\R^d)$ the Banach space of functions $f:\R^d\rightarrow\C$,
which are $p$-integrable with respect to the Lebesgue measure $\D x$
and, if $p=2$, the corresponding scalar product and norm are
$\langle\cdot,\cdot\rangle$ and $\|\cdot\|$, respectively. 
The Fourier trasform is denoted by $\mathcal F$ both on
$L^2(\R^d)$ and on  $L^1(\R^d)$, where it is 
defined by
\[
\widehat{f}(\xi)=\mathcal F f({\xi}\,)= \int_{\R^d} f(x) \E^{-2\pi i\,
  {\xi}\cdot x } \D{x},\qquad f\in L^1(\R^d).
\] 
If $G$ is a locally compact group, we denote by $L^2(G)$ the Hilbert
space of square-integrable functions with respect to a left Haar
measure on $G$.   We denote the (real) general linear group of size
  $d\times d$ by ${\rm
    GL}(d,\R)$ and by  ${\rm T}(d,\R)$ the closed subgroup of
  unipotent upper triangular  matrices.
 
If $H$ is a closed subgroup of ${\rm GL}(d,\R)$,  the
  semidirect product $G=\R^{d}\rtimes H$ is the product $\R^{d}\times
  H$ with  group operation
\[ 
(b_1,h_1)(b_2,h_2)=(b_1+h_1[b_2],h_1h_2),
\]
where $b_1, b_2\in\R^d$, $h_1,h_2\in H$ and where $h[b]$ is the natural linear  action of  the matrix $h$ on the column vector $b$.

\subsection{Shearlet dilation groups}
In this section we introduce the groups in which we are interested.
This family includes the groups introduced by F\"uhr in
\cite{fu98,futo16}, and called generalized shearlet dilation groups
for the purpose of generalizing the standard shearlet group
introduced in \cite{lalikuwe05,dastte10}.

\begin{defn}\label{GSDG} 
A  {\it shearlet dilation group}   $H<{\rm GL}(d,\R)$ is a subgroup of the form
$H=SD$, where 
\begin{enumerate}
\item[(i)]
$S$ is a Lie subgroup of ${\rm T}(d,\R)$ consisting of  matrices of the form 
\[
\begin{bmatrix}
1& -{^ts}\\0&B(s)
\end{bmatrix}
\]
with $s\in\R^{d-1}$ and $B:\R^{d-1}\to{\rm T}(d-1,\R)$ a smooth map;
\item[(ii)]  $D$ is the one-parameter subgroup of ${\rm GL}(d,\R)$ consisting of the diagonal matrices
\begin{equation}
a\,{\rm diag}(1,|a|^{\lambda_1}, \dots, |a|^{\lambda_{d-1}})
=a\begin{bmatrix}1&0\\0&\Lambda(a)\end{bmatrix}
\label{scale}\end{equation}
as $a$ ranges in $\R^{\times}$. Here  $(\lambda_1,\dots,\lambda_{d-1})$ is a fixed vector in $\R^{d-1}$.
\end{enumerate} 
The group $S$ is called the shearing sugroup of $H$ and $D$ is
called the diagonal complement or scaling subgroup of $H$.
\end{defn}
Several observations are in order. First of all, if one requires
  the shearing subgroup $S$ to be Abelian, then one obtains the class
  introduced by F\"uhr, with a slightly more general definition. This has inspired
  Definition~\ref{GSDG}.  

Since the map $B$ is
  continuous, $S$ is automatically connected, and hence by
  Theorem~3.6.2 in~\cite{raja84}, it is closed and simply
  connected. By  construction the elements of $H$ are of the form  
\begin{equation}
h_{s,a}= h_{s,1}h_{0,a}= 
a\begin{bmatrix}
1&- {^ts}\Lambda(a)\\0& B(s)\Lambda(a)
\end{bmatrix}. 
\label{hsa}
\end{equation}
Furthermore, since the diagonal matrices of ${\rm GL}(d,\R)$  normalize
  ${\rm T}(d,\R)$, then  $H$ is the semidirect product of $S$ and $D$. 
  
 Finally, the assumption that $S$ is a subgroup normalized  by
  $D$ forces the maps
  $B$ and $\Lambda$ to satisfy some equalities. Indeed, since
\[
\begin{bmatrix}
1&-{^tu}\\0&B(u)
\end{bmatrix}
\begin{bmatrix}
1&-{^tv}\\0&B(v)
\end{bmatrix}
=
\begin{bmatrix}
1&{-^t(v+{^tB(v)}u)}\\0&B(u)B(v)
\end{bmatrix},
\]
then $S$ is a group if and only if
\begin{align}
B(0) & =\operatorname{I}_{d-1} \label{BB} \\
B(u)B(v) & =B(v+{^tB(v)}u) \label{BBB} \\
B(u)^{-1}&=B(-{}^tB(u)^{-1}u) \label{BBBB}
\end{align}
for every $u,v\in\R^{d-1}$.   Since
\[
\begin{bmatrix}
1&0\\0&\Lambda(a)
\end{bmatrix}
\begin{bmatrix}
1&-{^ts}\\0&B(s)
\end{bmatrix}
\begin{bmatrix}
1&0\\0&\Lambda(a)^{-1}
\end{bmatrix}
=
\begin{bmatrix}
1&{^t(\Lambda(a)^{-1}s)}\\0&\Lambda(a)B(s)\Lambda(a)^{-1}
\end{bmatrix}
\]
the compatibility of $D$ with $S$ is equivalent to asking for the following condition to hold for all $a\not=0$ and all $s\in\R^{d-1}$:
\begin{equation}
\Lambda(a)B(s)\Lambda(a)^{-1}=B(\Lambda(a)^{-1}s).
\label{compatible}\end{equation}
It follows that $H$ is diffeomorphic as a manifold to $\R^{d-1} \times
\R^{\times}$, so that we can identify the element $h_{s,a}$ with the
pair $(s,a)$.  With this identification the product law amounts to 
\begin{equation}
(s,a) (s',a')=\left(\Lambda(a)^{-1}s'+{^tB(\Lambda(a)^{-1}s')}s,aa'\right).
\label{prodH}
\end{equation}
We stress that, in general, $S$ is not isomorphic as a Lie group 
to  the additive Abelian group $\R^{d-1}$, unless $S$ is the standard
shearlet group introduced in \cite{dastte10}, see the examples below.

\begin{remark}
  It should be clear that a slightly larger class would be obtained by
  allowing for diagonal matrices of the form
  \[ {\rm sign}(a)\,{\rm diag}(|a|^{\mu_0},|a|^{\mu_1}, \dots,
    |a|^{\mu_{d-1}}),
    \qquad
    a\in\R^\times.
  \]
  The case $\mu_0=0$, however, is uninteresting because any shearlet
  dilation group corresponding to this choice never admits admissible
  vectors \cite{futo16,ADDDF17}. But then a simple change of variables
  permits to assume $\mu_0=1$, as we did, and to set
  $\lambda_j=\mu_j-1$.
\end{remark}

\begin{remark} In \cite{futo16} the authors introduce the notion of shearlet dilation group 
by means of structural properties and then prove that in the case when $S$ is Abelian
they can be parametrized as in Definition~\ref{GSDG}.
\end{remark}

We now give three examples. If $S$ is Abelian a full characterization is
provided in \cite{futo16}, see also \cite{ADDDF17} for a connection
with a suitable class of subgroups of the symplectic group.
\begin{example}[The standard shearlet group]
  A possible choice for $B$ is the map $B(s)=\operatorname{I}_{d-1}$, which satisfies
  all the above properties. In this case, $s\mapsto h_{s,1}$ defines a
  group isomorphism between $\R^{d-1}$ and the Abelian group
  $S$. 

Clearly, any choice of the weights
  $\lambda_1,\ldots,\lambda_{d-1}$ is compatible
  with~\eqref{compatible}.  In particular, if we choose  as $D$ the group of matrices 
\[
A_a=a\left[\begin{matrix}1 & 0 \\ 0 & |a|^{\gamma-1}
    \operatorname{I}_{d-1}\end{matrix}\right] \qquad\Longleftrightarrow\qquad
  \Lambda(a)=|a|^{\gamma-1} \operatorname{I}_{d-1} \qquad
a\in\R^{\times}, 
\]
where $\gamma\in\R$ is a fixed parameter, then we obtain the
$d$-dimensional shearlet group, usually denoted  $\mathbb{S}^\gamma$, 
and, often, the parameter $\gamma$ is chosen
to be $1/d$ \cite{dastte10,ddgl15}.

\end{example}
\begin{example}[The Toeplitz shearlet group]
  Another important example arises when $B(s)$ is the Toeplitz matrix
  \begin{equation}
    B(s)=T(\hat{s}) 
    = 
    \begin{bmatrix} 
      1      & -s_1   & -s_2    & \ldots & -s_{d-2} \\
      0      & 1     & -s_1    & -s_2    & \vdots  \\
      \vdots &\ddots & \ddots & \ddots & \vdots  \\
      \vdots &       & \ddots & 1      & -s_1     \\
      0 & \dots & \ldots & 0 & 1
    \end{bmatrix},
    \label{toeplitz}\end{equation}
  where $\hat{s}={^t(s_1,\dots,s_{d-2})}$.  It is easy to see that
  $T(\hat{u})T(\hat{v})=T(\hat{u}\sharp\hat{v})$ where
  \begin{equation*}
    (\hat{u}\sharp \hat{v})_i 
    := 
    u_i+v_i + \sum_{j+k=i}v_j u_k, \quad
    i=1,\ldots d-2
  \end{equation*}
  and that consequently all the equalities in \eqref{BBB} hold. This
  case corresponds to Toeplitz shearlet groups (see~\cite{dddsst17}).

Not all dilation matrices as in \eqref{scale} are compatible with~\eqref{compatible}.
In \cite{futo16} it is shown that 
\begin{equation}
\lambda_k=k\lambda_1, 
\qquad
k=2,\dots,d-1
\label{topeq}\end{equation}
for any fixed $\lambda_1$.
\end{example}

\begin{example}[A non-Abelian shearlet dilation group] The matrices
\[
g(u_1,u_2,u_3)=
\begin{bmatrix}
1&-u_1&-u_2&-u_3\\
0&1&-u_1&-u_2-\frac{1}{2}u_1^2\\
0&0&1&0\\
0&0&0&1
\end{bmatrix}
\]
as $u=(u_1,u_2,u_3)$ ranges in $\R^3$ give rise to a non-Abelian shearlet group $S$. Indeed, it is easily checked that
\[
g(u_1,u_2,u_3)g(v_1,v_2,v_3)
=
g(u_1+v_1,u_2+v_2-u_1v_1,u_3+v_3-u_1(v_2+\frac{1}{2}v_1^2)),
\]
a  product which is not Abelian in the third coordinate. Evidently,
\[
B(u)=
\begin{bmatrix}
1&-u_1&-u_2-\frac{1}{2}u_1^2\\
0&1&0\\
0&0&1
\end{bmatrix}
\]
is a smooth function of $u$. The group $S$ is isomorphic to the standard Heisenberg group, as is most clearly seen at the level of Lie algebra. Indeed, the Lie algebra of $S$ is given by the matrices
\[
X(q,p,t)=
\begin{bmatrix}
0&q&p&t\\
0&0&q&p\\
0&0&0&0\\
0&0&0&0
\end{bmatrix},
\]
because $X^3(q,p,t)=0$ and hence
\begin{align*}
\exp(X(q,p,t))&=\operatorname{I}_4+X(q,p,t)+\frac{1}{2}X^2(q,p,t)\\
&=\begin{bmatrix}
1&q&p+\frac{1}{2}q^2&t+\frac{1}{2}qp\\
0&1&q&p\\
0&0&1&0\\
0&0&0&1
\end{bmatrix}\\
&=g(-q,-(p+\frac{1}{2}q^2),-(t+\frac{1}{2}qp)).
\end{align*}
Further,
\[
[X(q,p,t),X(q',p',t')]=X(0,0,qp'-pq')
\]
exhibits the Lie algebra of $S$ as the three dimensional Heisenberg Lie algebra. A straightforward calculation shows that
for any choice of $\lambda\in\R$ the diagonal matrices
\[
\Lambda(a)=
\begin{bmatrix}
|a|^{\lambda}&&\\
&|a|^{2\lambda}&\\
&&|a|^{3\lambda}
\end{bmatrix}
\]
normalize $B(u)$ because $\Lambda(a)B(u)\Lambda(a)^{-1}=B(\Lambda(a)^{-1}u)$. Conversely, these are easily seen to be the only rank-one dilations that normalize the matrices $B(u)$. In conclusion, the group $D$ consisting of the matrices
\[
a\begin{bmatrix}1&\\&\Lambda(a)\end{bmatrix}
\]
together with $S$ give rise to the non-Abelian shearlet dilation group $H=SD$. It is worth observing that the
dilations in $D$ are not the standard dilations of the Heisenberg group. Indeed, the Lie algebra of $D$ consists of the diagonal matrices 
$A_\lambda(\tau)={\rm diag}(\tau,(\lambda+1)\tau,(2\lambda+1)\tau,(3\lambda+1)\tau)$ and
\[
[A_\lambda(\tau),X(q,p,t)]=X(-\lambda\tau q,-2\lambda\tau p,-3\lambda\tau t)
\]
shows that these homogeneous dilations  are not  
the standard dilations of the Heisenberg Lie algebra 
(see \cite{ST}, p.~620).
\end{example}

\subsection{The shearlet representation and admissible vectors}

From now on we fix a group $G=\R^d\rtimes H$ where $H$
is a shearlet group as in Definition~\ref{GSDG} and we parametrize its elements
as  $(b,s,a)$. 
By~\eqref{prodH}  we get that a left Haar measure of
  $H$ is
\[ \D h= |a|^{\lambda_D-1}\D s\D a \]
where $\lambda_D=\lambda_1+\ldots+\lambda_{d-1}$ and $\D s$, $\D a$
are the Lebesgue measures of $\R^{d-1}$ and $\R^\times$. As a
consequence, a left Haar measure on $G$ is 
\[
\D g={\D b}\,\frac{\D h}{|\det h_{s,a}|} = |a|^{-(d+1)}\D b \D s\D a
\]
where $\D b$ is the Lebesgue measure on~$\R^d$ and the last equality
holds true since 
\[
|\det h_{s,a}|=|a|^{d+\lambda_D}.
\]
The quasi-regular representation of $G$ on $L^2(\R^d)$ is
\begin{equation}
S_{b,s,a}f(x)=|a|^{-\frac{d+\lambda_D}{2}} f(h_{s,a}^{-1}(x-b)).
\label{quasireg2}
\end{equation}

The next result generalizes Theorem 4.12 in \cite{futo16} to the case when $S$ is not Abelian.
\begin{theorem}\label{admvect}
The representation $S$ is square-integrable and its admissible vectors
$\psi$ are the elements of $L^2(\R^d)$ satisfying
\begin{equation}\label{eqn:admvect}
0< C_\psi=\int_{\R^d}\frac{|\cF{\psi}(\xi)|^2}{|\xi_1|^d}\D\xi<+\infty,
\end{equation}
where $\xi=(\xi_1,\xi')\in\R\times\R^{d-1}$.
\end{theorem}
We recall that a unitary representation $\pi$ of $G$ acting on a Hilbert
space $\mathcal H$  is   
{\it square integrable} if it is irreducible and if there exists a (non-zero)
element $\psi\in \mathcal H$, called 
{\it admissible} vector, such that the associated {\it voice
  transform}, {\em i.e.}  the linear map
$f\mapsto\langle f,\pi(b,h)\psi\rangle$,  takes values in
$L^2(G)$ and in such case it is a  multiple of an isometry, denoted  by $\cW_\psi:\mathcal H\to L^2(G)$.
\begin{proof}[Proof of Theorem~\ref{admvect}]
  The proof is an immediate consequence of the following result due to
F\"uhr, see \cite{fuhr10} and the references therein. The
  quasi-regular representation of $\R^d\rtimes H$ is square integrable
  if and only if there exists a vector $\xi_0\in\R^d$ such that
  \begin{enumerate}[(i)]
  \item the dual orbit
    $\mathcal{O}_{\xi_0}=\{{^th}{\xi_0}\in\R^d:h\in H\}$ is open and
    it is of full measure,
  \item  the stabilizer
    $H_{\xi_0}=\{h\in H: {^th}{\xi_0}=\xi_0\}$ is compact,
  \end{enumerate}
  where saying that $\mathcal{O}_{\xi_0}$ has full measure means that
  its complement has Lebsegue measure zero.  In such case, a vector $\psi$ is
  admissible if and only if
  \begin{equation}
    \label{eq:4}
    \int_{H} |\mathcal F\psi({}^th\xi_0)|^2 \D h <+\infty.
  \end{equation}
In our setting, with the choice $\xi_0=(1,0,\ldots,0)$ we have that
  \[ {}^th_{s,a}\xi_0={}^th_{0,a} {}^th_{s,1}
    \begin{bmatrix}
      1\\0\\ \ldots\\ 0
    \end{bmatrix}
    =a\begin{bmatrix} 1 \\ \Lambda(a) s
    \end{bmatrix}
  \]
  so that $\mathcal{O}_{\xi_0} = \R^\times \times\R^{d-1}$, which is of full
  measure, and $H_{\xi_0}$ is trivial. Hence $S$ is square-integrable.

  To compute the admissible vectors, notice that by~\eqref{eq:4}
  \begin{alignat*}{1}
    \int_{H} |\mathcal F\psi({}^th\xi_0)|^2 \D h& =
    \int_{\R^{d-1}\times\R^\times} |\mathcal F\psi(a\Lambda(a)s,a)|^2
    |a|^{\lambda_D-1} \D s \D a \\
    & =\int_{\R^{d-1}\times\R^\times} \dfrac{|\mathcal
      F\psi(\xi_1,\xi')|^2}{|\xi_1|^d} \D \xi_1 \D \xi'
  \end{alignat*}
  with the change of variables $a=\xi_1$ and
  $s=\Lambda(\xi_1)^{-1}\xi'/\xi_1$.
\end{proof}
Theorem~\ref{admvect} states the surprising fact that the
admissibility condition is the same for all generalized shearlet
dilation groups. A canonical choice is to assume that
\begin{equation}
  \label{eq:8}
\mathcal F\psi(\xi_1,\xi') =\mathcal F\psi_1(\xi_1) \mathcal F\psi_2(\xi'/\xi_1)
\end{equation}
where   $\psi_1\in L^2(\R)$ satisfies 
\begin{equation}
  \label{eq:9}
  \int_{\R^\times} \dfrac{|\mathcal
      F\psi_1(\xi_1)|^2}{|\xi_1|} \D \xi_1 <+\infty.
\end{equation}
and $\psi_2\in L^2(\R^{d-1})$.
However, other choices are available and, in particular,  it is
possible to build  shearlets with compact 
support in space \cite{kikuli12}.
We finally recall that, since the representation $S$ is
  square-integrable, we have the weakly-convergent reproducing formula
  \cite{fuhr05} 
\begin{equation}
\label{final}
f=\frac{1}{C_{\psi}}\int_{G}\mathcal{S}_{\psi} f(b,s,a)\,S_{b,s,a}\psi\ \frac{\D b\,\D s\,\D a}{|a|^{d+1}}.
\end{equation}

\subsection{Wavelet Transform}
We recall that the one-dimensional affine group $\mathbb{W}$ is $\R\rtimes\R^{\times}$ endowed with the product
\[
(b,a)(b',a')=(b+ab',aa')
\]
and left Haar measure $|a|^{-2}\D b\D a$. It acts on $L^2(\R)$ by
means of the square-integrable representation 
\[
W_{b,a}f(x)=|a|^{-\frac{1}{2}}f(\frac{x-b}{a}).
\]
The corresponding wavelet transform is $\mathcal{W}_{\psi}:L^2(\R)\rightarrow L^2(\mathbb{W})$, given by
\[
\mathcal{W}_{\psi}f(b,a)=\langle f,W_{b,a}\psi\rangle,
\]
which is a multiple of an isometry provided that $\psi\in L^2(\R)$ satisfies the admissibility condition, namely the Calder\'on equation,
\begin{equation}\label{calderon}
0<\int_{\R}\frac{|\mathcal{F}\psi(\xi)|^2}{|\xi|}\D \xi<+\infty
\end{equation}
and, in such a case, $\psi$ is called an admissible wavelet. 

\subsection{The quasi-regular representation of $H$}
Consider now the shearlet dilation group $H$,
with $S$ and $D$ its shearing and dilation subgroups, respectively (see Definition~\ref{GSDG}).
As mentioned above,  we identify $H$ with $\R^{d-1}\times\R^{\times}$ as manifolds
and sometimes denote by $(s,a)$ the element $h_{s,a}$ of $H$.
Recall that by $\eqref{prodH}$ the product law is then 
\[
(s,a)(s',a')=(\Lambda(a)^{-1}s'+{^tB(\Lambda(a)^{-1}s')}s,aa').
\]
Observe that $H$ acts naturally on $\R^d$ and its (right) dual action
  is 
\[
{}^th_{s,a}
\begin{bmatrix}
  v_1 \\ v
\end{bmatrix} = a \begin{bmatrix}
  v_1 \\  \Lambda(a)({^tB(s)}v-s)
\end{bmatrix}. 
\]
This implies that  $H$ acts naturally on $\mathbb
P^{d-1}=(\R^d\setminus\{0\})/\sim$ as well. By identifying $\R^{d-1}$ with 
$\{(1,v): v\in\R^{d-1}\}/\sim$ we get that $H$ acts on $\R^{d-1}$ as
\[ {}^th_{s,a}. v =   \Lambda(a)({^tB(s)}v-s).\]
Hence we can define the
quasi-regular representation of $H$ acting on $L^2(\R^{d-1})$ by means of 
\[
V_{s,a}f(v)=|a|^{\frac{\lambda_D}{2}}f(\Lambda(a)({^tB(s)}v-s)),
\]
where we recall that $\lambda_D=\lambda_1+\dots+\lambda_{d-1}$.  
In general, $V$ is not irreducible, but we can always define
the voice transform  associated to a fixed vector $\psi\in L^2(\mathbb{R}^{d-1})$, 
namely  the mapping $\mathcal{V}_{\psi}:L^2(\mathbb{R}^{d-1})\longrightarrow C(H)$ defined by 
\[
  \mathcal{V}_{\psi}f(s,a) =\langle f , V_{s,a}\psi\rangle_2,
\]
where $C(H)$ is the space of continuous functions on $H$.

\begin{example}[The standard shearlet group, continued]
  For the classical shearlet group $\mathbb S^\gamma$ the shearlet
  representation on $L^2(\R^d)$ becomes 
  \begin{equation}
S^\gamma_{b,s,a}f(x)=|a|^{-\frac{1+\gamma(d-1)}{2}}f(A_a^{-1}S_s^{-1}(x-b)),\label{eq:12}
\end{equation}
whereas  the group $H$ is the affine group $\R^{d-1}\rtimes \R^\times$ in dimension
  $d-1$ and $V$ is the corresponding wavelet representation
  \begin{equation}
V_{s,a}f(v)=|a|^{\frac{(d-1)(\gamma-1)}{2}}
f\left(\frac{v-s}{|a|^{1-\gamma}}\right),\label{eq:10}
\end{equation}
which is not irreducible unless $d=2$.  Furthermore, the voice
transform can be written as convolution operator
  \[
    \mathcal{V}_{\psi}f(s,a) = |a|^{\frac{(d-1)(\gamma-1)}{2}}  \int_{\R^{d-1}}
    f(v) \overline{\psi\left(\frac{v-s}{|a|^{1-\gamma}}\right) }
    \D v = f * \Psi_a(s)
  \]
  where
  \[\Psi_a(v)=|a|^{\frac{(d-1)(\gamma-1)}{2}} 
    \overline{\psi\left(-\frac{v}{|a|^{1-\gamma}}\right)} .\]
\end{example}

\subsection{The affine Radon transform}\label{affineRadon}
In this section we recall the definition and the main properties of the Radon transform. Then we introduce the particular restriction of the Radon transform in which we are interested, the so-called {\it affine Radon transform}, obtained by parametrizing the space of hyperplanes by affine coordinates.

We first define the Radon transform on $L^1(\R^d)$ by following
  the approach in~\cite{raka96}, see also \cite{helgason99} as a
  classical reference. Given $f\in
  L^1(\R^d)$ its Radon transform is the function $\mathcal{R}f  :(\R^d\setminus\{0\})\times\R\rightarrow\C$
  defined by
\begin{equation}
\mathcal{R}f(n,t) =\frac{1}{|n|}\int_{n\cdot x=t}f(x)\ {\rm d}m(x), \label{radon}
\end{equation}
where $m$ is the Euclidean measure on the hyperplane
\begin{equation}
(n:t):=\{x\in\R^d:n\cdot x=t\}\label{eq:5}
\end{equation}
and the equality~\eqref{radon} holds for almost all $(n,t)\in
(\R^d\setminus\{0\})\times\R$. 
We add some comments. 
Definition~\eqref{radon} makes sense since, given  $n\in
\R^d\setminus\{0\}$ Fubini theorem gives that 
\[
\int_{\R^d}|f(x)|\D x = \int_{\R} \left(\int_{n\cdot x=t}|f(x)|\D
  m(x)\right)\D t <+\infty,
\]  
so that for almost all $t\in \R$ the integral  
$\int_{n\cdot x=t}|f(x)|\D m(x)$ is finite and $\mathcal{R}(n,t)$ is well defined.

Furthermore, each pair
$(n,t)\in(\mathbb{R}^d\setminus\{0\})\times\R$ defines the hyperplane $(n:t)$ 
by means of~\eqref{eq:5}. Clearly, the
correspondence between parameters $(n,t)$ and hyperplanes  is not
bijective. Indeed $(n',t')$ and $(n,t)$ determine the same hyperplane
if and only if there exists $\lambda\in\R^\times$ such that
$n'=\lambda n$ and $t'=\lambda t$ and this equivalence relation
motivates the notation $(n:t)$ for the hyperplane~in~\eqref{eq:5}.
Because of the factor $1/|n|$ in~\eqref{radon}, 
$\mathcal{R}f$ is a positively homogenous function of degree $-1$,
{\em i.e.} for all  $\lambda\in\R^\times$
\begin{equation}
  \label{eq:6}
\mathcal{R}f(\lambda n, \lambda t)=|\lambda|^{-1}\mathcal{R}f(n,t).
\end{equation}
This means  that $\mathcal{R}f$ is completely defined by choosing a representative
$(n,t)$ for each hyperplane $(n:t)$, {\em i.e.} by choosing a suitable system of
coordinates on the affine Grassmannian
\[ \{\text{hyperplanes of } \R^d\}\simeq \mathbb{P}^{d-1}\times\R.\]
The canonical choice \cite{helgason99} is given by parametrizing
$\mathbb{P}^{d-1}$ with its two-fold covering  $S^{d-1}$,  where
$S^{d-1}$ is the unit sphere in $\R^d$.

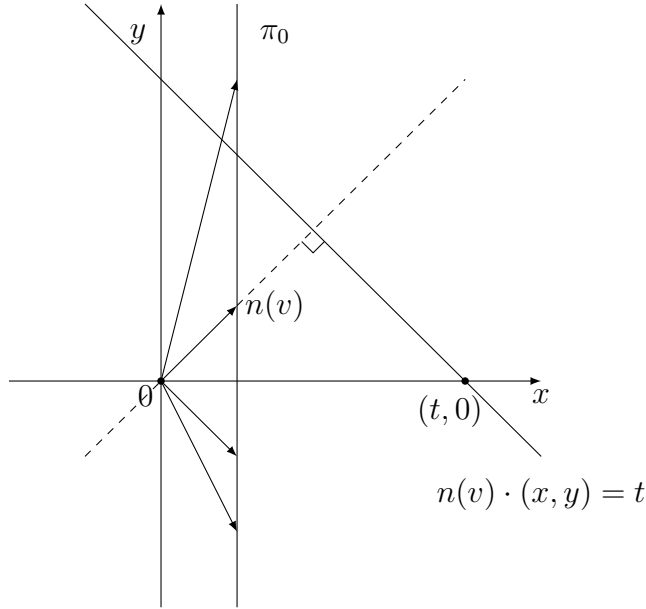
\begin{figure}
\centering
\begin{tikzpicture} [>=latex]
\draw plot [domain=-1:5] (\x, {-\x+4)}); 
\draw [dashed] (-1,-1) - - (0,0); 
\draw [dashed] (1,1) - - (4,4); 
\draw [-] (1,-3) -- (1,5); 
\node at (1.5,4.6) {$\pi_0$}; 
\draw [->] (-2,0) -- (5,0); 
\node at (5,-0.2) {$x$}; 
\draw [->] (0,-3) -- (0,5); 
\node at (-0.3,4.6) {$y$}; 
\fill (0,0) circle (1.4pt); 
\node at (-0.2,-0.2) {0}; 
\node at (1.5,1) {$n(v)$}; 
\draw [->] (0,0) -- (1,1);
\draw [->] (0,0) -- (1,4);
\draw [->] (0,0) -- (1,-1);
\draw [->] (0,0) -- (1,-2);
\node at (5,-1.5) {$n(v)\cdot(x,y)=t$}; 
\node at (3.8,-0.4) {$(t,0)$}; 
\fill (4,0) circle (1.4pt); 
\draw (2,2) coordinate (b);
\draw[anchor=base,color=black]  (b.center) ++(-0.15,-0.15)  -- ++(0.15,-0.15) -- ++(0.15,0.15);
\end{tikzpicture}
\caption{space of hyperplanes parametrized by affine coordinates (2-dimensional case)}
\label{fig:radonaffine}
\end{figure}
We are interested in another restriction of the Radon transform.  For
all $v\in\mathbb{R}^{d-1}$ set $^tn(v)=(1,{^tv})$.
\begin{defn}
Given $f\in L^1(\R^d)$, the {\it affine Radon transform of $f$} is the function
$\mathcal{R}^{\text{aff}}f:\R^{d-1}\times\R\rightarrow\C$
given by
\begin{align}
\label{affine}
\mathcal{R}^{\text{aff}}f(v,t)&=\mathcal{R}f(n(v),t)\nonumber\\
&=\frac{1}{\sqrt{1+|v|^2}}\int_{n(v)\cdot x=t}f(x){\rm d}m(x)=\int_{\mathbb{R}^{d-1}}f(t-v\cdot y,y){\rm d}y.
\end{align}
\end{defn}
\begin{remark}
  The transform $\mathcal{R}^{\rm aff}$ is obtained from $\mathcal{R}$
  by parametrizing the projective space ${\mathbb P}^{d-1}$ with
    affine coordinates. Indeed, the map $(v,t)\mapsto(n(v):t)$ is a
    diffeomorphism of $\R^{d-1}\times\R$ onto the open subset
    \[ 
    U_0=\{(n:t): \exists\lambda\in\R^\times\text{ s.t. } \lambda
      n\in\pi_0\},
      \] 
      where $\pi_0=\{n(v):v\in\mathbb{R}^{d-1}\}$.
      The complement of $U_0$ is the set of {\em
      horizontal hyperplanes},  those for which the normal vector has the first
    component equal to zero (see Figure \ref{fig:radonaffine} for the
    2-dimensional case). The set of pairs $(v,t)$ such that
    $(n(v):t)\not\in U_0$ is negligible, so that
    $\mathcal{R}^{\text{aff}}f$ completely defines $\mathcal Rf$.
    In~\ref{polaff} we recall the relation between the affine Radon
    transform and the usual Radon transform in polar coordinates.
  \end{remark}

The next proposition, whose proof  can be found in \cite{raka96},
summarizes the behaviour of the 
Radon transform under affine linear actions.  The translation and dilation operators act on a
function $f:\R^d\to \C$ as
\[
T_bf(x)=f(x-b),
\qquad
D_Af(x)=|\det A|^{-1}f\left(A^{-1}x\right),
\]
respectively, for $b\in\R^d$ and $A\in\text{GL}(d,\R)$. Both operators
map each $L^p(\R^d)$ onto itself and  $D_A$ is normalized to be
  an isometry on $L^1(\R^d)$.  
\begin{proposition}
\label{radonprop}
Given $f\in L^1(\R^d)$, the following properties hold true:
\begin{enumerate}
\item[(i)] $\mathcal{R}T_{b}f(n,t)=\mathcal{R}f(n,t-n\cdot b)$,\ for all $b\in\mathbb{R}^d$;
\item[(ii)] $\mathcal{R}D_{A}f(n,t)=\mathcal{R}f({^t\!A}n,t)$,\ for all $A\in GL(d,\mathbb{R})$.
\end{enumerate}
\end{proposition}

We  now state a crucial result in Radon transform theory in its standard version. 
Below we prove two variations that are taylored to our setting but are also of some independent interest.
\begin{proposition}[Fourier slice theorem, 1]
\label{FSTG}
For any  $f\in L^1(\mathbb{R}^d)$  
\[
\mathcal{F}(\mathcal{R}f(n,\cdot))(\tau)=\mathcal{F}f(\tau n).
\]
for all $n\in\mathbb{R}^d\setminus\{0\}$ and all $\tau\in\mathbb{R}$.
\end{proposition} 
Here the Fourier transform on the right hand side is in $\R^d$, whereas the
operator $\mathcal F$ on the left hand side is 1-dimensional and acts on the
variabile~$t$. We repeat this slight abuse of notation in other
formulas below.

In the next formulation, written for the affine Radon transform, the function $f$ to which $ \mathcal{R}^{\rm aff}$ is applied is taken in $L^1(\R^d)\cap L^2(\R^d)$.
\begin{proposition}[Fourier slice theorem, 2]\label{fst}
Define $\psi:\R^{d-1}\times(\R\setminus\{0\})\rightarrow\R^d$ by  $\psi(v,\tau)=\tau n(v)$. For every  $f\in L^1(\R^d)\cap L^2(\R^d)$ there exists a negligible set $E\subseteq\R^{d-1}$ such that for all $v\not\in E$ the function $\mathcal{R}^{\rm aff}f(v,\cdot)$ is in $L^2(\mathbb{R})$ and satisfies
\begin{equation}
\label{fst2}
\mathcal{R}^{\rm aff}f(v,\cdot)=\mathcal{F}^{-1}[\mathcal{F}f\circ\psi(v,\cdot)].
\end{equation}
\end{proposition}
\begin{proof}
By Proposition \ref{FSTG} we know that for all $v\in\R^{d-1}$ the affine Radon transform $\mathcal{R}^{\rm aff}f(v,\cdot)$ is in $L^1(\R)$ and satisfies
\[
\mathcal{F}(\mathcal{R}^{\rm aff}f(v,\cdot))(\tau)=\mathcal{F}f\circ\psi(v,\tau),\qquad \tau\in\R.
\]
We start by  proving that the function $\tau\mapsto\mathcal{F}f\circ\psi(v,\tau)$ is in $L^2(\R)$, that is 
\[
\int_{\R}|\mathcal{F}f\circ\psi(v,\tau)|^2\D \tau<+\infty.
\]
The map $\psi:\mathbb{R}^{d-1}\times(\mathbb{R}\setminus\{0\})\rightarrow\mathbb{R}^d$, defined by $\psi(v,\tau)=\tau n(v)$, is a diffeomorphism onto the open set $V=\{\xi\in\mathbb{R}^d:\xi_1\ne0\}$ with Jacobian $J\psi(v,\tau)=\tau^{d-1}$.
By hypothesis we know that 
\begin{align*}
\|f\|_2^2=\int_{\R^d}|\mathcal{F}f(\xi)|^2\D \xi=\int_{\R^{d-1}}\int_{\R}|\mathcal{F}f\circ\psi(v,\tau)|^2|\tau|^{d-1}\D\tau\D v<+\infty,
\end{align*}
so that there exists a negligible set $E\subseteq\R^{d-1}$ such that 
\[
C_f:=\int_{\R}|\mathcal{F}f\circ\psi(v,\tau)|^2|\tau|^{d-1}\D\tau<+\infty
\]
for all $v\not\in E$. Therefore, for all $v\not\in E$ it holds 
\begin{align*}
\int_{\R}|\mathcal{F}f\circ\psi(v,\tau)|^2\D
  \tau & =\int_{|\tau|\leq1}|\mathcal{F}f\circ\psi(v,\tau)|^2\D
  \tau+\int_{|\tau|>1}\frac{|\tau|^{d-1}}{|\tau|^{d-1}}|\mathcal{F}f\circ\psi(v,\tau)|^2\D
  \tau\\
&\leq  2 \lVert \mathcal{F}f\rVert_{\infty} +  \int_{\R}|\mathcal{F}f\circ\psi(v,\tau)|^2|\tau|^{d-1}\D\tau\\ 
&\leq 2\|f\|_1^2+ C_f<+\infty.
\end{align*}
Hence the function $t\mapsto\mathcal{R}^{\rm aff}f(v,t)$ is in
$L^1(\R^d)\cap L^2(\R^d)$ and \eqref{fst2} follows by  the Fourier inversion
formula in $L^2(\R)$.
\end{proof}

It is possible to extend the affine Radon transform $\mathcal{R}^{\rm aff}$ to $L^2(\R^d)$ as a unitary map. However, this raises some technical issues, that are addressed in the next section.
\subsection{The unitary extension}
Consider the subspace
\[
\mathcal{D}=\bigl\{f\in L^2(\mathbb{R}^{d-1}\times\mathbb{R}):
\int_{\mathbb{R}^{d-1}\times\mathbb{R}}|\tau|^{d-1}\left|\mathcal{F}f(\xi,\tau)\right|^2\,\D\xi\D\tau<+\infty\bigr\}
\]
of $L^2(\mathbb{R}^{d-1}\times\mathbb{R})$ and define the operator $\mathcal{I}:\mathcal{D}\rightarrow L^2(\mathbb{R}^{d-1}\times\mathbb{R})$ by
\begin{equation} 
\mathcal{F}\mathcal{I}f(\xi,\tau)=|\tau|^\frac{d-1}{2}\mathcal{F}f(\xi,\tau),
\label{riesz}
\end{equation}
a Fourier multiplier with respect to the last variable. 
Since $\tau\mapsto |\tau|^\frac{d-1}{2}$ is a strictly positive
  (almost everywhere) Borel
  function on $\R$, the spectral theorem for unbounded operators, see Theorem~VIII.6 of
\cite{resi80}, shows that $\mathcal{D}$ is dense and that  $\mathcal{I}$ is a
positive self-adjoint injective operator.

\begin{remark}
  The operator $\mathcal{I}$ is related to the inverse of the Riesz
  potential with exponent $(d-1)/2$ on $L^2(\mathbb{R})$. Indeed, if
  $\psi_2\in L^2(\mathbb{R}^{d-1})$ and if $\psi_1\in L^2(\mathbb{R})$
  is such that
  \[
    \int_{\mathbb{R}}|\tau|^{d-1}|\mathcal{F}\psi_1(\tau)|^2\ {\rm
      d}\tau<+\infty,
  \]
  then $\psi_2\otimes\psi_1\in\mathcal{D}$, because
  \begin{align*}
    &\int_{\mathbb{R}^{d-1}\times\mathbb{R}}|\tau|^{d-1}|\mathcal{F}(\psi_2\otimes\psi_1)(\xi,\tau)|^2\ {\rm d}\xi {\rm d}\tau\\
    &=\int_{\mathbb{R}^{d-1}}|\mathcal{F}\psi_2(\xi)|^2\ {\rm d}\xi\ \int_{\mathbb{R}}|\tau|^{d-1}|\mathcal{F}\psi_1(\tau)|^2\ {\rm d}\tau<+\infty,
  \end{align*}
  so that
  \[
    \mathcal{I}(\psi_2\otimes\psi_1)=\psi_2\otimes\mathcal{I}_0\psi_1,
  \]
  where $\mathcal{I}_0$ is the inverse of the standard Riesz potential
  defined by
  \begin{equation}
    \label{Izero}
    \mathcal{F}\mathcal{I}_0\psi_1(\tau)=|\tau|^\frac{d-1}{2}\mathcal{F}\psi_1(\tau).
  \end{equation}
\end{remark}


Furthermore, $\mathcal{D}$ is invariant under translations and dilations by matrices of the form
\begin{equation}
\label{matrixform}
A=\left[\begin{matrix}A_0 & 0 \\ v & a\end{matrix}\right],
\end{equation}
where $A_0\in\text{GL}(d-1,\R),\ v\in\mathbb{R}^{d-1},\ a\in\mathbb{R}^\times$.
\begin{lemma}
\label{operatorI}
For all $b\in\mathbb{R}^d$ and $A$ as in \eqref{matrixform} it holds
\begin{equation}
\label{Irelations}
\mathcal{I}T_b=T_b\mathcal{I},\qquad \mathcal{I}D_A=|a|^{-\frac{d-1}{2}}D_A\mathcal{I}.
\end{equation}

\end{lemma}
\begin{proof}
  The first of relations \eqref{Irelations} is a consequence of the
  fact that
  $\mathcal{F}T_bf(\xi,\tau)=e^{-2\pi
    ib\cdot\xi}\mathcal{F}f(\xi,\tau)$ for all
  $f\in L^2(\mathbb{R}^{d-1}\times\mathbb{R})$.  Precisely, for all
  $f\in\mathcal{D}$ we have that
\begin{align*}
\mathcal{F}\mathcal{I}T_bf(\xi,\tau)&=|\tau|^\frac{d-1}{2}\mathcal{F}T_bf(\xi,\tau)\\
&=|\tau|^\frac{d-1}{2}e^{-2\pi ib\cdot\xi}\mathcal{F}f(\xi,\tau)\\
&=e^{-2\pi ib\cdot\xi}\mathcal{F}\mathcal{I}f(\xi,\tau)\\
&=\mathcal{F}T_b\mathcal{I}f(\xi,\tau),
\end{align*}
whence $\mathcal{I}T_b=T_b\mathcal{I}$. The second follows from 
$\mathcal{F}D_Af(\xi,\tau)=\mathcal{F}f({^t\!A}(\xi,\tau))$. Indeed, for all $f\in\mathcal{D}$
\begin{align*}
\mathcal{F}\mathcal{I}D_Af(\xi,\tau)&=|\tau|^\frac{d-1}{2}\mathcal{F}D_Af(\xi,\tau)\\
&=|\tau|^\frac{d-1}{2}\mathcal{F}f({^t\!A}(\xi,\tau))\\
&=|\tau|^\frac{d-1}{2}\mathcal{F}f({^t\!A}_0\xi+\tau v,a\tau)\\
&=|\tau|^\frac{d-1}{2}|a\tau|^{-\frac{d-1}{2}}\mathcal{F}\mathcal{I}f({^t\!A}_0\xi+\tau v,a\tau)\\
&=|a|^{-\frac{d-1}{2}}\mathcal{F}D_A\mathcal{I}f(\xi,\tau).
\end{align*}
This proves \eqref{Irelations}.
\end{proof} 
The space $\mathcal{D}$ becomes a pre-Hilbert space with respect to the scalar product
\begin{equation}
\label{scalarproduct}
\begin{split}
\langle f,g\rangle_\mathcal{D}&=\langle \mathcal{I}f,\mathcal{I}g\rangle_2\\
&=\int_{\mathbb{R}^{d-1}\times\mathbb{R}}|\tau|^{d-1}\mathcal{F}(f(v,\cdot))(\tau)\overline{\mathcal{F}(g(v,\cdot))(\tau)}\ {\rm d}v{\rm d}\tau.
\end{split}
\end{equation}
Furthermore,
\[
||f||^2_{\mathcal{D}}=\langle f,f\rangle_{\mathcal{D}}=\langle\mathcal{I}f,\mathcal{I}f\rangle_2=||\mathcal{I}f||_2^2,
\]
for all $f\in\mathcal{D}$. Hence $\mathcal{I}$ is an isometric operator from $\mathcal{D}$, with the new scalar product \eqref{scalarproduct}, to $L^2(\mathbb{R}^{d-1}\times\mathbb{R})$. Since $\mathcal{I}$ is self-adjoint and injective,  $\text{Ran}(\mathcal{I})$ is dense in $L^2(\mathbb{R}^{d-1}\times\mathbb{R})$. Hence, by standard arguments, it extends uniquely to a unitary operator, denoted $\mathscr{I}$, from the  completion $\mathcal{H}$ of  $\mathcal{D}$ onto $L^2(\mathbb{R}^{d-1}\times\mathbb{R})$.\\ 

To extend $\mathcal{R}^{\rm aff}$ to $L^2(\R^d)$ as a unitary
operator, note that, by Proposition~\ref{FSTG} with $n=n(v)$, the affine Radon transform of $f\in L^1(\R^d)\cap L^2(\R^d)$ belongs to $L^2(\R^{d-1}\times\R)$ if and only if is finite the integral
\begin{align*}
\int_{\R^{d-1}\times\R}|\mathcal{R}^{\rm aff}f(v,t)|^2 \D v\D t&=\int_{\R^{d-1}}\int_{\R}|\mathcal{F}(\mathcal{R}^{\rm aff}f(v,\cdot))(\tau)|^2 \D \tau\D v\\
&=\int_{\R^{d-1}\times\R}|\mathcal{F}f(\tau, \tau v)|^2 \D \tau\D v\\
&=\int_{\R^d}\frac{|\mathcal{F}f(\xi)|^2}{|\xi_1|^{d-1}}\D\xi,
\end{align*}
where $\xi_1$ is the first component of the vector $\xi\in\R^d$. Therefore requiring that $\mathcal{R}^{\rm aff}f$ belongs to $L^2(\R^{d-1}\times\R)$ is equivalent to 
\[
\int_{\R^d}\frac{|\mathcal{F}f(\xi)|^2}{|\xi_1|^{d-1}}\D\xi<+\infty.
\]
We denote by
\[
\mathcal{A}=\{f\in L^1(\R^d)\cap L^2(\mathbb{R}^d) : \int_{\R^d}\frac{|\mathcal{F}f(\xi)|^2}{|\xi_1|^{d-1}}\D\xi<+\infty\},
\]
which is dense in $L^2(\R^d)$ since it contains the functions  whose
Fourier transform  is smooth  and  has   compact support
disjoint from the hyperplane $\xi_1=0$. 
By definition of $\mathcal{A}$,  $\mathcal{R}^{\rm aff}f\in L^2(\R^{d-1}\times\R)$ for all $f\in\mathcal{A}$. 
 

We shall need a suitable formulation of the main result in Radon
transform theory, namely the following version of Theorem~4.1 in
\cite{helgason99}.  For the sake of completeness we include the proof
  in~\ref{other-proofs}.

\begin{theorem}
\label{tfstg}
The affine  Radon transform extends to a unique unitary operator from
$L^2(\mathbb{R}^d)$ onto $\mathcal{H}$, denoted with $\mathscr{R}$
and, hence, $\cQ=\mathscr{I}\mathscr{R}$ is a unitary operator from
$L^2(\mathbb{R}^d)$ onto $L^2(\mathbb{R}^{d-1}\times\mathbb{R})$. 
\end{theorem}

As mentioned above, we need yet another generalization of the Fourier slice theorem (Proposition~\ref{FSTG}). We think that it is perhaps known,
  but we could not locate it in the literature. The proof is given in~\ref{other-proofs}.
\begin{proposition}[Fourier slice theorem, 3]\label{fstg3}
For all $f\in L^2(\mathbb{R}^d)$ 
\begin{equation}
\label{fstg}
\mathcal{F}(\cQ f(v,\cdot))(\tau)=|\tau|^\frac{d-1}{2}\mathcal{F}f(\tau n(v)) 
\end{equation}
for almost every $(v,\tau)\in\R^{d-1}\times\R$.
\end{proposition}

\section{The Intertwining Theorem and its consequences}\label{main}

\subsection{The main Theorem}
We recall that the group $G$ is the semidirect product $G=\mathbb{R}^{d}\rtimes H$ where 
$H=SD$ is the  shearlet dilation group,  $S$ is the shearing subgroup 
and $D$ the scaling subgroup of $H$, as in Definition~\ref{GSDG}. Each
element in $G$ is
parametrized by a triple
$(b,s,a)\in\R^d\times\R^{d-1}\times\R^{\times}$ and $S_{b,s,a}$ is as in \eqref{quasireg2}.

\begin{theorem}
\label{generalteo}
The  unitary operator
$\cQ$ intertwines the shearlet representation with the tensor product of two unitary representations, precisely
\begin{equation}
\label{general}
\cQ S_{b,s,a}f(v,t)=(V_{s,\, a}\otimes W_{n(v)\cdot b,a})\cQ f(v,t)
\end{equation}
for every $f\in L^2(\mathbb{R}^d)$. \end{theorem}
\begin{proof}
By density, it is enough to prove the equality on $\mathcal{A}$. We shall use throughout the fact that $\mathcal{R}^{\rm aff}f\in\cD$ for every $f\in\mathcal{A}$ and that $\cD$ is invariant under all translations and under the dilations described in \eqref{matrixform}. Since for all $(b,s,a)\in G$ it holds $(b,s,a)=(b,0,1)(0,s,1)(0,0,a)$, it is sufficient to prove the equality for each of the three factors. For $f\in\mathcal{A}$ and $b\in\mathbb{R}^d$ we have
\begin{align*}
\mathcal{R}^{\rm aff}S_{b,0,1}f(v,t)&=\mathcal{R}^{\rm aff}T_{b}f(v,t)\\
&=\mathcal{R}T_{b}f(n(v),t)\\
&=\mathcal{R}f(n(v),t-n(v)\cdot b)\\
&=\mathcal{R}^{\rm aff}f(v,t-n(v)\cdot b)\\
&=(\operatorname{I}\otimes W_{n(v)\cdot b,1})\mathcal{R}^{\rm aff}f(v,t).
\end{align*}
Since $\mathcal{I}$ commutes with translations,  $\operatorname{I}\otimes W_{n(v)\cdot b,1}=T_{(0,n(v)\cdot b)}$ implies
\[
\mathcal{I}\mathcal{R}^{\rm aff}S_{b,0,1}f(v,t)=\mathcal{I}(\operatorname{I}\otimes W_{n(v)\cdot b,1})\mathcal{R}^{\rm aff}f(v,t)=(\operatorname{I}\otimes W_{n(v)\cdot b,1})\mathcal{I}\mathcal{R}^{\rm aff}f(v,t).
\]
For $f\in\cA$ and $a\in\R^{\times}$ we have
\begin{align*}
\mathcal{R}^{\rm
  aff}S_{0,0,a}f(v,t)&=|a|^{\frac{d+ \lambda_D}{2}}\mathcal{R}^{\rm
                       aff}D_{h_{0,a}}f(v,t)\\ 
&=|a|^{\frac{d+\lambda_D}{2}}\mathcal{R}D_{h_{0,a}}f(n(v),t)\\
&=|a|^{\frac{d+\lambda_D}{2}}\mathcal{R}f(^{t}h_{0,a}n(v),t).
\end{align*}
A direct calculation gives
\[
^{t}h_{0,a}n(v)=a\left[\begin{matrix}1 & 0 \\ 0 & \Lambda(a) \end{matrix}\right]\left[\begin{matrix}1\\ v\end{matrix}\right]=a\left[\begin{matrix}1\\ {\Lambda(a)}{v }\end{matrix}\right]=an\left({\Lambda(a)}{v }\right).
\]
The behavior of the Radon transform under linear operations implies that
\begin{align*}
\mathcal{R}^{\rm aff}S_{0,0,a}f(v,t)
&=|a|^{\frac{d+\lambda_D}{2}}\mathcal{R}f\left(an\left({\Lambda(a)}{v }\right),a\frac{t}{a}\right)\\
&=|a|^{\frac{d+\lambda_D}{2}-1}\mathcal{R}^{\text{aff}}f\left({\Lambda(a)}v,\frac{t}{a}\right)\\
&=|a|^{\frac{d-1}{2}}(V_{0,\,a}\otimes W_{0,a})\mathcal{R}^{\rm aff}f\left(v,t\right).
\end{align*}
Since
\[
(V_{0,\,a}\otimes W_{0,a})=|a|^\frac{3(1-\lambda_D)}{2}D_{A},
\]
where the matrix $A$ is of the form
\[
A=\left[\begin{matrix}\Lambda(a)^{-1} & 0 \\ 0 & a \end{matrix}\right],
\]
and because of the behavior of the operator $\mathcal{I}$ under dilations, we obtain 
\begin{align*}
\mathcal{I}\mathcal{R}^{\rm aff}S_{0,0,a}f(v,t)&=\mathcal{I}|a|^{\frac{d-1}{2}}(V_{0,\,a}\otimes W_{0,a})\mathcal{R}^{\rm aff}f(v,t)\\
&=(V_{0,\,a}\otimes W_{0,a})\mathcal{I}\mathcal{R}^{\rm aff}f(v,t).
\end{align*}
Finally, let $s={^t(s_1,\ldots,s_{d-1})}\in\R^{d-1}$. Then
\begin{align*}
\mathcal{R}^{\rm aff}S_{0,s,1}f(v,t)&=\mathcal{R}^{\rm aff}D_{h_{s,1}}f(v,t)\\
&=\mathcal{R}D_{h_{s,1}}f(n(v),t)=\mathcal{R}f(^{t}h_{s,1}n(v),t).
\end{align*}
Since
\[
^{t}h_{s,1}n(v)=\left[\begin{matrix}1 & 0\\ -s & ^{t}B(s)\end{matrix}\right]\left[\begin{matrix}1\\ v\end{matrix}\right]=\left[\begin{matrix}1\\ {^{t}B(s)}v-s \end{matrix}\right]=n\left({^{t}B(s)}v-s\right),
\]
by Proposition~\ref{radonprop} we obtain the following string of equalities:
\begin{align*}
\mathcal{R}^{\rm aff}S_{0,s,1}f(v,t)
&=\mathcal{R}f\left(n\left({^{t}B(s)}v-s\right),t\right)\\
&=\mathcal{R}^{\rm aff}f\left({^{t}B(s)}v-s,t\right)\\
&=(V_{s,1}\otimes \operatorname{I})\mathcal{R}^{\rm aff}f\left(v,t\right).
\end{align*}
Finally,
\[
(V_{s,1}\otimes \operatorname{I})=|a|^\frac{3(1-\lambda_D)}{2}T_{(-({^{t}B(s)})^{-1}s,0)}D_{A},
\]
where 
\[
A=\left[\begin{matrix}{^{t}B(s)}^{-1} & 0 \\ 0 & 1 \end{matrix}\right],
\]
so that the behavior of $\mathcal{I}$ under dilations implies
\begin{align*}
\mathcal{I}\mathcal{R}^{\rm aff}S_{0,s,1}f(v,t)
&=\mathcal{I}(V_{s,1}\otimes \operatorname{I})\mathcal{R}^{\rm aff}f(v,t)\\
&=(V_{s,1}\otimes \operatorname{I})\mathcal{I}\mathcal{R}^{\rm aff}f(v,t).
\end{align*}
Therefore, by
\[
\mathcal{I}\mathcal{R}^{\rm aff}S_{b,s,a}f=\mathcal{I}\mathcal{R}^{\rm aff}S_{b,0,1}S_{0,s,1}S_{0,0,a}f,
\]
equation \eqref{general} follows applying the relations obtained above.
\end{proof}

\subsection{The admissibility conditions}
In this subsection we discuss the admissibility conditions and some of their consequences. 

Our objective is to obtain an expression for the shearlet transform that makes use of formula~\eqref{general}.
To this end, we start by looking for  natural conditions that
guarantee that $\psi\in L^2(\R^d)$ is  an admissible vector for
  the shearlet representation $S$, namely that it  satisfies \eqref{eqn:admvect}. 

Equation \eqref{general} suggests that a good choice for the
admissible vector $\psi$ is of the form  
\[
\cQ\psi=\phi_2\otimes\phi_1
\]
where $\phi_1\in L^2(\R)$, $\phi_2\in L^2(\R^{d-1})$. If this is the
case, then by~\eqref{fstg} it follows that 
\[
\phi_2(v)\mathcal{F}\phi_1(\tau)=|\tau|^\frac{d-1}{2}\mathcal{F}\psi(\tau n(v)) 
\]
so that $\mathcal{F}\psi$ factorizes as
\begin{equation}
\label{condition}
\mathcal{F}\psi(\tau,\tau v)=\mathcal{F}\psi_1(\tau)\mathcal{F}\psi_2(v),
\end{equation}
where we assume that $\psi_2\in L^2(\mathbb{R}^{d-1})$ and $\psi_1\in
L^2(\mathbb{R})$. Equation~\eqref{condition} is  the canonical choice of admissible
vectors given by~\eqref{eq:8}. Furthermore,
\[
\mathcal{F}\phi_1(\tau)=|\tau|^\frac{d-1}{2}\mathcal{F}\psi_{1}(\tau)\qquad \phi_2(v)=\mathcal{F}\psi_2(v),
\]
so that the assumption that $\psi_2\in L^2(\R^{d-1})$ is automatically satisfied.
Since $\phi_1\in L^2(\R)$, then    
\[
\int_{\mathbb{R}}|\tau|^{d-1}|\mathcal{F}\psi_1(\tau)|^2\ {\rm d}\tau<+\infty.
\]
This, together with the fact that $\psi_1\in L^2(\R)$, implies that  $\psi_1$ belongs to the domain of the differential operator
$\mathcal{I}_0$ (see~ \eqref{Izero}).
Therefore
\begin{equation}
\phi_1=\mathcal{I}_0\psi_1.
\label{pippo}
\end{equation}
 With the choice~\eqref{condition} the
admissibility condition~\eqref{eqn:admvect} reduces to  
 \[
0<\int_{\R}\frac{|\mathcal{F}\psi_1(\tau)|^2}{|\tau|}\D \tau<+\infty.\]

From now on we fix $\psi\in L^2(\mathbb{R}^d)$ of the form
\eqref{condition} with  $\psi_1\in L^2(\mathbb{R})$ satisfying 
\begin{equation}
\label{conditions}
\int_{\mathbb{R}}|\tau|^{d-1}|\mathcal{F}\psi_1(\tau)|^2\ {\rm d}\tau<+\infty,
\qquad 
0<\int_{\R}\frac{|\mathcal{F}\psi_1(\tau)|^2}{|\tau|}\D \tau<+\infty,
\end{equation}
and $\psi_2\in L^2(\mathbb{R}^{d-1})$.

\begin{corollary}\label{formulazza}
Under the assumptions \eqref{conditions}, for every $L^2(\R^d)$ 
\begin{align}
\label{firstm}
\mathcal{S}_{\psi}f(b,s,a)
=\mathcal{V}_{\phi_2}\left(\mathcal{W}_{\phi_1}(\cQ f(v,t))(n(v)\cdot b,a)\right)(s,a)
\end{align}
\end{corollary}
\begin{proof}
For all $f\in L^2(\mathbb{R}^d)$ and $(b,s,a)\in G$
\begin{align}
\label{first}
\nonumber&\mathcal{S}_{\psi} f(b,s,a)=\langle f,S_{b,s,a}\psi\rangle_2\\
\nonumber&=\langle \cQ f,\cQ S_{b,s,a}\psi\rangle_2\\
\nonumber&=\langle \cQ f,(V_{s, a}\otimes W_{n(\cdot)\cdot b,a})\cQ\psi\rangle_2\\
\nonumber&=\langle \cQ f,(V_{s, a}\otimes W_{n(\cdot)\cdot b,a})(\phi_2\otimes\phi_1)\rangle_2\\
\nonumber&=\langle \cQ f,V_{s, a}\phi_2\otimes W_{n(\cdot)\cdot b,a}\phi_1\rangle_2\\
\nonumber&=\int_{\mathbb{R}^{d-1}\times\mathbb{R}}\cQ f(v,\tau)\overline{V_{s, a}\phi_2(v)W_{n(v)\cdot b,a}\phi_1(\tau)}\ {\rm d}v{\rm d}\tau\\
\nonumber&=\int_{\mathbb{R}^{d-1}}\left(\int_{\mathbb{R}}\cQ f(v,\tau)\overline{W_{n(v)\cdot b,a}\phi_1(\tau)}\ {\rm d}\tau\right)\overline{V_{s, a}\phi_2(v)}\ {\rm d}v\\
&=\int_{\mathbb{R}^{d-1}}\mathcal{W}_{\phi_1}(\cQ
  f(v,\bullet))(n(v)\cdot b,a)\overline{V_{s, a}\phi_2(v)}\ {\rm d}v,
\end{align}
where in the last  equality we have used the fact that $\phi_1$ is an admissible wavelet. This is true  because 
by \eqref{pippo}
\begin{align*}
\int_{\mathbb{R}}\frac{|\mathcal{F}\phi_1(\tau)|^2}{|\tau|}\ {\rm d}\tau&=\int_{\mathbb{R}}\frac{|\mathcal{F}\mathcal{I}_0\psi_1(\tau)|^2}{|\tau|}\ {\rm d}\tau\\
&\leq\int_{0<|\tau|<1}\frac{|\mathcal{F}\psi_1(\tau)|^2}{|\tau|}\ {\rm d}\tau+\int_{|\tau|\geq1}|\tau|^{d-1}|\mathcal{F}\psi_1(\tau)|^2\ {\rm d}\tau\\
&\leq\int_{\mathbb{R}}\frac{|\mathcal{F}\psi_1(\tau)|^2}{|\tau|}\ {\rm d}\tau+\int_{\mathbb{R}}|\tau|^{d-1}|\mathcal{F}\psi_1(\tau)|^2\ {\rm d}\tau,
\end{align*}
which are both finite.
\end{proof}

Equation~\eqref{first} shows that the shearlet coefficients
  $\mathcal{S}_{\psi} f(b,s,a)$ can be computed in terms of  the
  unitary Radon transform $\cQ f$, which involves the
  pseudo-differential operator $\mathcal I$ and it is difficult to
  compute numerically.  However, if $f\in L^1(\R^d)\cap L^2(\R^d)$,
there is yet a different way to express the shearlet transform. To this end we need to choose $\psi$ in such a way that $\cQ\psi$ is in the domain of the operator $\mathcal{I}$, that is, in such a way that
\[
\int_{\mathbb{R}^{d-1}\times\mathbb{R}}|\tau|^{d-1}|\mathcal{F}\cQ\psi(v,\tau)|^2\
{\rm d}v{\rm d}\tau<+\infty. 
\] 
Assuming this and recalling that $\cQ\psi=\mathcal{F}\psi_2\otimes\cI_0\psi_1$ we obtain
\begin{align*}
&\int_{\mathbb{R}^{d-1}\times\mathbb{R}}|\tau|^{d-1}|\mathcal{F}\cQ\psi(v,\tau)|^2\ {\rm d}v{\rm d}\tau\\
&=\int_{\mathbb{R}^{d-1}\times\mathbb{R}}|\tau|^{d-1}|\mathcal{F}(\mathcal{F}\psi_2\otimes\mathcal{I}_0\psi_1)(v,\tau)|^2\ {\rm d}v{\rm d}\tau\\
&=\int_{\mathbb{R}^{d-1}}|\mathcal{F}\psi_2(v)|^2\ {\rm d}v\ \int_{\mathbb{R}}|\tau|^{d-1}|\mathcal{F}\mathcal{I}_0\psi_1(\tau)|^2\ {\rm d}\tau\\
&=||\psi_2||_2^2\ \int_{\mathbb{R}}|\tau|^{2(d-1)}|\mathcal{F}\psi_1(\tau)|^2\ {\rm d}\tau.
\end{align*}
This shows that $\cQ\psi$ is in the domain of $\mathcal{I}$ if and
only if $\psi_1$ satisfies the additional condition  
\begin{equation}
\label{newcondition}
\int_{\mathbb{R}}|\tau|^{2(d-1)}|\mathcal{F}\psi_1(\tau)|^2\ {\rm d}\tau<+\infty.
\end{equation}
In this case, by \eqref{pippo}
\[ 
\mathcal{I} Q\psi = \mathcal{I} (\phi_2\otimes\phi_1)=
  \phi_2\otimes \mathcal{I}_0\phi_1. 
  \]
  \begin{corollary}\label{formulazza2}
Under the assumptions \eqref{conditions} and \eqref{newcondition},  
\begin{align}
\nonumber&\mathcal{S}_{\psi} f(b,s,a)=|a|^{-\frac{d-1}{2}}\mathcal{V}_{\phi_2}\left(\mathcal{W}_{\chi_1}
(\mathcal{R}^{\rm aff}f(v,t))(n(v)\cdot b,a)\right)(s,a).
\end{align}
for all $f\in L^1(\R^d)\cap L^2(\R^d)$.
\end{corollary}
\begin{proof}
For all $f\in L^1(\R^d)\cap L^2(\R^d)$ and $(b,s,a)\in G$
\begin{align}
\label{secondo}
\nonumber\mathcal{S}_{\psi}f(b,s,a)&=\int_{\mathbb{R}^{d-1}}\left(\int_{\mathbb{R}}\cQ f(v,\tau)\overline{W_{n(v)\cdot b,a}\phi_1(\tau)}\ {\rm d}\tau\right)\overline{V_{s, a}\phi_2(v)}\ {\rm d}v\\
&=\int_{\mathbb{R}^{d-1}} \langle \cQ f(v,\cdot),
           W_{n(v)\cdot b,a} \phi_1\rangle_2 \overline{V_{s,
           a}\phi_2(v)}\ {\rm d}v .
\end{align}

Since $ f\in L^1(\R^d)\cap L^2(\R^d)$, Proposition~\ref{fst}
and~Proposition~\ref{fstg3}  imply that for almost all $v\in\R^{d-1}$,  
$\mathcal{R}^{\rm aff}f(v,\cdot)$ is in $L^2(\R)$ and
\begin{alignat*}{1}
   \mathcal{F}(\cQ
  f(v,\cdot))(\tau) & =|\tau|^\frac{d-1}{2}\mathcal{F}f(\tau,\tau v) = |\tau|^\frac{d-1}{2}  \mathcal{F}\mathcal{R}^{\rm
    aff}f(v,\cdot)(\tau).
\end{alignat*}
Since $\mathcal{F}(\cQ f(v,\cdot))\in L^2(\R)$ for almost all $v\in\R^{d-1}$, the above equality
  implies that $\mathcal{R}^{\rm aff}f(v,\cdot)$ is in the domain of
$\mathcal{I}_0$ and, by definition of $\mathcal{I}_0$,
\[
  \cQ f(v,\cdot)(\tau)  = \mathcal{I}_0 \mathcal{R}^{\rm aff}f(v,\cdot).
\]
By assumption $\phi_1$ is  in the domain of $\mathcal{I}_0$ and the
same property holds true for  $W_{n(v)\cdot b,a}\phi_1$. Since  $\mathcal{I}_0$ is
self-adjoint, we get 
  \begin{alignat*}{1}
    \langle \cQ f(v,\cdot), W_{n(v)\cdot b,a} \phi_1\rangle_2 & =
    \langle \mathcal{R}^{\rm aff}f(v,\cdot) , \mathcal{I}_0
    W_{n(v)\cdot b,a} \phi_1\rangle_2 \\
& = |a|^{-\frac{d-1}{2}} \langle
    \mathcal{R}^{\rm aff}f(v,\cdot) , W_{n(v)\cdot b,a}
    \mathcal{I}_0\phi_1\rangle_2 ,
  \end{alignat*}
by taking  into account that
\[
\mathcal{I}_0 W_{n(\cdot)\cdot b,a} =|a|^{-\frac{d-1}{2}}  W_{n(\cdot)\cdot b,a}\mathcal{I}_0.
\]
Setting $\chi_1=\mathcal{I}_0\phi_1=\mathcal{I}^2_0\psi_1$,  {\em i.e.}
\begin{equation}\label{chi}
\mathcal{F}\chi_1(\tau)=|\tau|\mathcal{F}\psi_1(\tau),
\end{equation}
from~\eqref{secondo} we finally get 
\begin{align}
  \label{second}
 & \mathcal{S}_{\psi}f(b,s,a)=\\
&|a|^\frac{\lambda_D+1-d}{2}\int_{\mathbb{R}^{d-1}}\mathcal{W}_{\chi_1}(\mathcal{R}^{\rm
    aff}f(v,\cdot))(n(v)\cdot
  b,a)\overline{\phi_2\left(\Lambda(a)({^tB(s)}v-s)\right)}\ {\rm
    d}v \nonumber
\end{align}
Observe that  we have used the fact that $\chi_1$ is an admissible
wavelet, too, the proof is analogous to the proof that $\psi_1$ is
such.  As for~\eqref{first} we can rewrite the above formula by using
the voice transform of $H$, {\em i.e.} 
\begin{align}
\nonumber&\mathcal{S}_{\psi} f(b,s,a)=|a|^{-\frac{d-1}{2}}\mathcal{V}_{\phi_2}\left(\mathcal{W}_{\chi_1}
(\mathcal{R}^{\rm aff}f(v,t))(n(v)\cdot b,a)\right)(s,a).
\end{align}

\end{proof}

Observe that formulas \eqref{first} and \eqref{second} can be also written in terms of the polar Radon transform using relation \eqref{relationpolaff}. 

Equation \eqref{second} shows that for any signal $f\in L^1(\R^d)\cap
  L^2(\R^d)$ the shearlet coefficients can be computed by means of
  three {\it classical} transforms: first compute the affine Radon
  transform $\mathcal{R}^{\rm aff}f$, then apply the wavelet transform
  to the last variable 
\[
G(v,b,a)=\mathcal{W}_{\chi_1}(\mathcal{R}^{\rm aff}f(v,\cdot))(n(v)\cdot b,a),
\]
where $\chi_1$ is given by \eqref{chi}, and, finally, ``mock-convolve'' with respect to the variable $v$
\[
\mathcal{S}_{\psi}^\gamma f(b,s,a)= \int_{\mathbb{R}^{d-1}} G(v,b,a) 
\Phi_a(s-{}^t\!B(s)v) \ {\rm d}v 
\]
with  the scale-dependent filter 
\[
\Phi_a(v)=\overline{\phi_2\left(- \Lambda(a)v\right)}.
\]
Note that the ``mok-convolution'' reduces to the standard convolution
in $\R^{d-1}$ when $B(t)=\operatorname{I}_{d-1}$.

Notice that the shearlet coefficients $\mathcal{S}_{\psi}^\gamma f(b,s,a)$
depend on $f$ only through its affine Radon transform
$\mathcal{R}^{\rm aff}f$. Therefore Equation~\eqref{final} allows to
reconstruct any unknown signal $f\in L^1(\R^d)\cap L^2(\R^{d})$ from its Radon
transform by computing the shearlet coefficients by means of
\eqref{second}. Finally, it is worth observing that this reconstruction does not involve
the differential operator $\mathcal{I}$ as applied to the signal.
Hence,  another interesting aspect of our result is that it could open the way
to new methods for inverting the Radon transform, a very important
issue in applications. Indeed, this result leads to an inversion
formula for the Radon transform based on the shearlet and the wavelet
transforms.

\begin{example}[The standard shearlet group, continued]
For the classical shearlet group $\mathbb S^\gamma$, \eqref{general} becomes
\begin{equation}
\label{eqn:intertwing}
\cQ S^\gamma_{b,s,a}f(v,t)=(V_{s,a}\otimes W_{n(v)\cdot b,a})\cQ f(v,t),
\end{equation}
where $S^\gamma_{b,s,a}$ is given by~\eqref{eq:12} and 
$V_{s,a}$ is the wavelet representation in  dimension $d-1$  as in~\eqref{eq:10}. 
Therefore in the case of the standard shearlet group
Theorem~\ref{generalteo} shows that the unitary operator $\cQ$
intertwines the shearlet representation $S^\gamma$ with the tensor
product of two wavelet representations. 

For a fixed admissible vector $\psi\in L^2(\mathbb{R}^d)$ of the form \eqref{condition} with $\psi_1\in L^2(\mathbb{R})$ satisfying \eqref{conditions} and $\psi_2\in L^2(\mathbb{R}^{d-1})$, equation \eqref{first} becomes 
\begin{align}\label{first2}
&\nonumber\mathcal{S}^\gamma_{\psi}f(b,s,a)\\
&=|a|^\frac{(d-1)(\gamma-1)}{2}\int_{\mathbb{R}^{d-1}}\mathcal{W}_{\phi_1}(\cQ f(v,\bullet))(n(v)\cdot b,a)\overline{\phi_2\left(\frac{v-s}{|a|^{1-\gamma}}\right)}\ {\rm d}v,
\end{align}
for any $f\in L^2(\R^d)$ and $(b,s,a)\in G$.
Assuming that $\psi_1$ satisfies the additional condition
\eqref{newcondition}, for any $f\in L^1(\R^d)\cap L^2(\R^d)$ and $(b,s,a)\in G$, equality
\eqref{second} becomes  
\begin{align}\label{second2}
&\nonumber\mathcal{S}^\gamma_{\psi}f(b,s,a)\\
&=|a|^\frac{(d-1)(\gamma-2)}{2}\int_{\mathbb{R}^{d-1}}\mathcal{W}_{\chi_1}(\mathcal{R}^{\rm aff}f(v,\bullet))(n(v)\cdot b,a)\overline{\phi_2\left(\frac{v-s}{|a|^{1-\gamma}}\right)}\ {\rm d}v,
\end{align}
where $\chi_1$ is the admissible vector defined by \eqref{chi}.

For the sake of clarity we write the above equation for $d=2$ and in terms
of the Radon transform in polar coordinates
  \begin{align*}
& \mathcal{S}^\gamma_{\psi}f(x,y,s,a) \\& 
= |a|^{\frac{\gamma-2}{2}}\int_{\R}   \mathcal W_{\chi_1} \left( \mathcal{R}^{\rm pol}
  f(\arctan v,\frac{\bullet }{ \sqrt{1+v^2}} )\right)(x+v y,a)\, \overline{
  \phi_2
  \Big(\frac{v-s}{|a|^{1-\gamma}}\Big)}\,\frac{\D v}{\sqrt{1+v^2}}, 
  \end{align*}
where $x,y,s\in \R$, $a\in\R^\times$ and $f\in L^1(\R^d)\cap L^2(\R^2)$.  As
mentioned in the introduction,  $\chi_1=\mathcal{I}^2_0\psi_1$,
see~\eqref{pippo}, so that the admissible $1D$-wavelet $\chi_1$ is
proportional to the Hilbert
transform $H$ of the weak derivative of $\psi_1$, which is the first
factor of the shearlet admissible vector $\widehat{\psi}(\xi_1,\xi_2)=
\widehat{\psi_1}(\xi_1) \widehat{\psi_2}(\xi_2/\xi_1)$. 
\end{example}

\section*{Acknowledgement} 
F. De Mari and E. De Vito are members of the Gruppo Nazionale per l'Analisi
  Matematica, la Probabilit\`a e le loro Applicazioni (GNAMPA) of the
  Istituto Nazionale di Alta Matematica (INdAM). 

\appendix

\section{The polar and the affine Radon transform}\label{polaff}
As mentioned in Section~\ref{affineRadon}, the natural restriction of the
Radon transform is the {\it polar Radon transform}
$\mathcal{R}^{\text{pol}}f$, which is obtained by restricting
$\mathcal{R}f$ to the closed subset $S^{d-1}\times\R$, where $S^{d-1}$
is the unit sphere in $\R^d$. Define
$\Theta^{d-1}=[0,\pi]^{d-2}\times[0,2\pi)$. For all
$\theta\in\Theta^{d-1}$ we write inductively  
\[
  ^t\theta={(\theta_1,{^t\hat{\theta}})},\qquad \theta_1\in[0,\pi],\
  \hat{\theta}\in\Theta^{d-2}
\]
and then we put 
\[
^t\eta(\theta)={(\cos\theta_1,\sin\theta_1{^t{\eta}(\hat{\theta})})},
\]
where $\eta(\hat{\theta})\in S^{d-2}$ corresponds to the previous inductive step. Clearly, the map $\eta:\Theta^{d-1}\rightarrow S^{d-1}$ induces  a parametrization of the unit sphere in $\mathbb{R}^d$.
Also, observe that the map $\Theta^{d-1}\rightarrow\mathbb{P}^{d-1}$ given by
$(\theta,t)\mapsto(\eta(\theta):t)$
is a two-fold covering of  \
$\mathbb{P}^{d-1}$.  
\begin{defn}
Take $f\in L^1(\R^d)$. The {\it polar Radon transform} of $f$ is the function 
$\mathcal{R}^{\rm pol}f:\Theta^{d-1}\times\mathbb{R}\rightarrow\mathbb{C}$
defined by
\begin{equation}
\label{pol}
\mathcal{R}^{\rm pol}f(\theta,t)
=\mathcal{R}f(\eta(\theta),t)
=\int_{\eta(\theta)\cdot x=t}f(x)\ {\rm d}m(x).
\end{equation}
\end{defn}
It is easy to find the relation between  $\mathcal{R}^{\rm pol}$ and $\mathcal{R}^{\rm aff}$. Using the parametrization $\eta$  of the unit sphere, we can write any vector $n(v)$ as $n(v)=\sqrt{1+|v|^2}\ \eta(\theta)$. More precisely, there exists ${^t\theta}=(\theta_1,{^t\hat{\theta}})\in\Theta^{d-1}$ such that 
\begin{equation}
\label{polar}
(1,{^tv})=\sqrt{1+|v|^2}(\cos\theta_1,\sin\theta_1{^t{\eta}(\hat{\theta})}).
\end{equation}
Equality \eqref{polar} holds if and only if 
\[
\cos\theta_1=\frac{1}{\sqrt{1+|v|^2}},\qquad \sin\theta_1{\eta}(\hat{\theta})=\frac{v}{\sqrt{1+|v|^2}}.
\]
It follows that
\[
\theta_1=\arccos\Bigl(\frac{1}{\sqrt{1+|v|^2}}\Bigr)\in[0,\frac{\pi}{2}),\qquad {\eta}(\hat{\theta})=\frac{v}{|v|},
\]
unless $v=0$, in which case ${\eta}(\hat{\theta})$ can be any vector in $S^{d-2}$.
Then, item (i) of Proposition~\ref{radonprop} gives that 
\begin{equation}
\label{relationpolaff}
\mathcal{R}^{\text{aff}}f(v,t)=\frac{1}{\sqrt{1+|v|^2}}\mathcal{R}^{\text{pol}}f\Bigl(\theta,\frac{t}{\sqrt{1+|v|^2}}\Bigr)
\end{equation}
and this is the relation that we need.

\section{Other proofs}\label{other-proofs}

\begin{proof}[Proof of Theorem~\ref{tfstg}]
For the reader's convenience we adapt the proof of \cite{helgason99} to our context.
Recall that  the map $\psi:\mathbb{R}^{d-1}\times(\mathbb{R}\setminus\{0\})\rightarrow\mathbb{R}^d$, defined by $\psi(v,\tau)=\tau n(v)$, is a diffeomorphism onto the open set $V=\{\xi\in\mathbb{R}^d:\xi_1\ne0\}$ with Jacobian $J\Phi(v,\tau)=\tau^{d-1}$. Thus, the Plancherel theorem and the Fourier slice theorem give that 
for any $f\in\mathcal{A}$
\begin{align*}
||f||^2_2
&=\int_V|\mathcal{F}f(\xi)|^2\ {\rm d}\xi\\
&=\int_{\mathbb{R}^{d-1}\times(\mathbb{R}\setminus\{0\})}|\mathcal{F}f(\tau n(v))|^2|\tau|^{d-1}\ {\rm d}v{\rm d}\tau\\
&=\int_{\mathbb{R}^{d-1}\times\mathbb{R}}|\mathcal{F}(\mathcal{R}^{\rm aff}f(v,\cdot))(\tau)|^2|\tau|^{d-1}\ {\rm d}v{\rm d}\tau\\
&=||\mathcal{R}^{\rm aff}f||^2_{\mathcal{D}}.
\end{align*}
Thus, $\mathcal{R}^{\rm aff}f$ belongs to $\mathcal{D}$ for all $f\in\mathcal{A}$ and $\mathcal{R}^{\rm aff}$ is an isometric operator from $\mathcal{A}$ into $\mathcal{D}$. We want to prove that $\mathcal{R}^\text{aff}:\mathcal{A}\rightarrow\mathcal{H}$ has dense image in $\mathcal{H}$. Since $\mathcal{H}$ is the completion of $\mathcal{D}$, it is enough to prove that $\mathcal{R}^\text{aff}$ has dense image in $\mathcal{D}$, that is (Ran$(\mathcal{R}^\text{aff})$)$^\perp=\{0\}$ in $\mathcal{D}$. Take then $\varphi\in\mathcal{D}$ such that  
$\langle\varphi,\mathcal{R}^\text{aff}f\rangle_{\mathcal{D}}=0$
for all $f\in\mathcal{A}$. By the definition of the scalar product on $\mathcal{D}$ and the Fourier  slice theorem we have that
\begin{align*}
\langle\varphi,\mathcal{R}^\text{aff}f\rangle_{\mathcal{D}}&=\int_{\mathbb{R}^{d-1}\times\mathbb{R}}|\tau|^{d-1}\mathcal{F}(\varphi(v,\cdot))(\tau)\overline{\mathcal{F}(\mathcal{R}^\text{aff}f(v,\cdot))(\tau)}\ {\rm d}v{\rm d}\tau\\
&=\int_{\mathbb{R}^{d-1}\times\mathbb{R}}|\tau|^{d-1}\mathcal{F}(\varphi(v,\cdot))(\tau)\overline{\mathcal{F}f(\tau n(v))}\ {\rm d}v{\rm d}\tau\\
&=\int_{\mathbb{R}^{d}}\mathcal{F}\Bigl(\varphi\bigl(\frac{\tilde{\xi}}{\xi_1},\cdot\bigr)\Bigr)(\xi_1)\overline{\mathcal{F}f(\xi)}{\rm d}\xi,
\end{align*}
where ${^t\xi}=(\xi_1,{^t\tilde{\xi}})$.  Therefore, if $\langle\varphi,\mathcal{R}^\text{aff}f\rangle_{\mathcal{D}}=0$
for all $f\in\mathcal{A}$, then 
\[
\mathcal{F}\Bigl(\varphi\bigl(\frac{\tilde{\xi}}{\xi_1},\cdot\bigr)\Bigr)(\xi_1)=0
\]
almost everywhere. However,
\[
\|\varphi\|_{\mathcal{D}}=\int_{\mathbb{R}^{d-1}\times\mathbb{R}}|\tau|^{d-1}|\mathcal{F}(\varphi(v,\cdot))(\tau)|^2\ {\rm d}v{\rm d}\tau
=\int_{\mathbb{R}^{d}}|\mathcal{F}\Bigl(\varphi\bigl(\frac{\tilde{\xi}}{\xi_1},\cdot\bigr)\Bigr)(\xi_1)|^2\ {\rm d}\xi
\]
and  hence $\varphi=0$ in $\mathcal{D}$. Therefore $\mathcal{R}^\text{aff}:\mathcal{A}\rightarrow\mathcal{H}$ has dense image in $\mathcal{H}$ and we can extend  it  to a unique unitary operator $\mathscr{R}$ from $L^2(\mathbb{R}^d)$ onto $\mathcal{H}$. Hence, $\cQ=\mathscr{I}\mathscr{R}$ is a unitary operator from $L^2(\mathbb{R}^d)$ onto $L^2(\mathbb{R}^{d-1}\times\mathbb{R})$.
\end{proof}

\begin{proof}[Proof of Proposition~\ref{fstg3}]
We start by observing that \eqref{fstg} is true if $f\in\mathcal{A}$, by the Fourier slice theorem and by the definition of $\cQ$. Take now  $f\in L^2(\mathbb{R}^d)$. By density there exists a sequence $(f_n)_n\in\mathcal{A}$ such that $f_n\rightarrow f$ in $L^2(\mathbb{R}^d)$. Since $\cQ$ is unitary from $L^2(\mathbb{R}^d)$ onto $L^2(\mathbb{R}^{d-1}\times\mathbb{R})$ and $\operatorname{I}\otimes\mathcal{F}$ is unitary from $L^2(\mathbb{R}^{d-1}\times\mathbb{R})$ into itself, where $\operatorname{I}$ is the identity operator, $(\operatorname{I}\otimes\mathcal{F})\cQ f_n\rightarrow(\operatorname{I}\otimes\mathcal{F})\cQ f$ in $L^2(\mathbb{R}^{d-1}\times\mathbb{R})$. 
Since $f_n\in\mathcal{A}$, for almost every $(v,\tau)\in\R^{d-1}\times\R$
\begin{align*}
(\operatorname{I}\otimes\mathcal{F})\cQ f_n(v,\tau)&=(\operatorname{I}\otimes\mathcal{F})\mathcal{I}\mathcal{R}^{\rm aff}f_n(v,\tau)\\
&=|\tau|^\frac{d-1}{2}(\operatorname{I}\otimes\mathcal{F})\mathcal{R}^{\rm aff}f_n(v,\tau)\\
&=|\tau|^\frac{d-1}{2}\mathcal{F}f_n(\tau n(v)).
\end{align*}
So that, passing to a subsequence if necessary, \[|\tau|^\frac{d-1}{2}\mathcal{F}f_n(\tau n(v))\rightarrow(\operatorname{I}\otimes\mathcal{F})\cQ f(v,\tau)\] for almost every $(v,\tau)\in\R^{d-1}\times\R$. Therefore for almost every $(v,\tau)\in\R^{d-1}\times\R$,
\begin{align*}
(\operatorname{I}\otimes\mathcal{F})\cQ f(v,\tau)=\lim_{n\rightarrow+\infty}|\tau|^\frac{d-1}{2}\mathcal{F}f_n(\tau n(v))=|\tau|^\frac{d-1}{2}\mathcal{F}f(\tau n(v)),
\end{align*}
where the last equality holds true using a subsequence if necessary.
\end{proof}


\end{document}